\documentclass[amssymb,amsfonts,reqno]{amsart}
\usepackage{color}
\usepackage{latexsym}
\usepackage{amsmath}
\usepackage{amssymb}
\usepackage{amsfonts}
\usepackage{esint}
\usepackage{graphicx}
\usepackage{accents}

\newcommand{\der}{\nabla}

\newcommand{\bea}{\begin{eqnarray}}
\newcommand{\eea}{\end{eqnarray}}
\newcommand{\wht}[1]{\widetilde{#1}}

\def\pa{\partial}

\def\beaa{\begin{eqnarray*}}
\def\eeaa{\end{eqnarray*}}
\def\pa{\partial}
\def\a{{\alpha}}
\def\b{{\beta}}

\def\ep{\epsilon}

\def\si{\sigma}

\def\th{\theta}

\def\nn{{\mathbb N}}

\def\R{{\mathbb R}}

\def\bar{\overline}

\def\qed{$\Box$\medskip}
\def \p{ \partial}

\def\12{\frac{1}{2}}

\def\one{\bbbone}
\def\cH{{\mathcal H}}
\def\cA{{\mathcal A}}

\def\cB{{\mathcal B}}

\def\I{{\mathcal I}}
\def\N{{\mathcal N}}

\def\bep{\begin{proposition}}
\def\eep{\end{proposition}}

\def\4{\frac{1}{4}}

\def\cE{\mathcal{E}}

\def\cN{{\mathcal N}}

\def\qed{$\Box$\medskip}
\def \p{ \partial}

\def\12{\frac{1}{2}}

\def\one{\bbbone}
\def\cH{{\mathcal H}}

\def\I{{\mathcal I}}
\def\N{\nn}

\def\bep{\begin{proposition}}
\def\eep{\end{proposition}}

\def\qed{\hfill \raisebox{0.5ex}{\framebox[1.6ex]{
                                       \rule[0ex]{0ex}{0.3ex} }}}

\def\build#1_#2^#3{\mathrel{\mathop{\kern 0pt#1}\limits_{#2}^{#3}}}
\def\4{\frac{1}{4}}

\def\cE{\mathcal{E}}

\def\cN{{\mathcal N}}

\def\<{\langle}
\def\>{\rangle}
\def\one{{\mathchoice {\rm 1\mskip-4mu l} {\rm 1\mskip-4mu l} {\rm 1\mskip-4.5mu l} {\rm
1\mskip-5mu l}}}

\theoremstyle{plain}
\newtheorem{theorem}{Theorem}
\newtheorem{proposition}{Proposition}
\newtheorem{lemma}{Lemma}
\newtheorem{corollary}{Corollary}

\theoremstyle{remark}
\newtheorem{remark}{Remark}

\theoremstyle{definition}
\newtheorem{definition}{Definition}
\numberwithin{equation}{section}
\numberwithin{proposition}{section}
\numberwithin{definition}{section}
\numberwithin{lemma}{section}
\numberwithin{corollary}{section}
\numberwithin{remark}{section}
\begin{document}
\title[Yang-Mills fields on the Schwarzschild black hole]{Instability
  of infinitely-many stationary solutions of the $SU(2)$ Yang-Mills fields on the exterior of the Schwarzschild black hole}
\author{Dietrich H\"afner, C\'ecile Huneau}
\address{Universit\'e Grenoble-Alpes, Institut Fourier, 100 rue des
  maths, 38610 Gi\`eres, France}
\email{Dietrich.Hafner@univ-grenoble-alpes.fr}
\email{Cecile.Huneau@univ-grenoble-alpes.fr}

\maketitle

\begin{abstract} 
We consider the spherically symmetric $SU(2)$ Yang-Mills fields on the
Schwarzschild metric. Within the so called purely magnetic Ansatz
we show that there exists a countable number of stationary solutions which are all nonlinearly
unstable.  
\end{abstract}

\setcounter{page}{1}
\pagenumbering{arabic}

\section{Introduction}

\subsection{General introduction}

We study the $SU(2)$ Yang-Mills equations on the Schwarzschild
metric, with spherically symmetric initial data fulfilling the so
called purely magnetic Ansatz. This
equation has at least a countable number of stationary solutions. In \cite{GH}
the first author and S. Ghanem showed that the zero curvature
solution is stable within this Ansatz. In this paper we show that the other solutions of this set are
nonlinearly unstable.  

Global existence for Yang-Mills fields on $\R^{3+1}$ was shown by
Eardley and Moncrief in a classical result, \cite{EM1} and
\cite{EM2}. Their result was then generalized by Chru\'sciel and
Shatah to general globally hyperbolic curved space-times in
\cite{CS}. Later, the hypotheses of \cite{CS} were weakened in
\cite{G1}. 

The purely magnetic Ansatz excludes  Coulomb type solutions and
reduces the Yang-Mills equations to a nonlinear scalar wave equation:
\begin{equation}
\label{SYM}
\pa_{t}^{2} {W}- \pa_{x}^{2} W+ \frac{(1- \frac{2m}{r}) }{r^2} W(W^2-1)=0.
\end{equation}

Strong numerical evidence of the existence of a countable number of stationary solutions  $(W_n)_{n\in \N}$ in the case of Yang Mills equations coupled with Einstein equations with spherical symmetry was shown in \cite{cbh} (see also \cite{BRZ}).
It was then proved analytically, still in the coupled case, in \cite{SWY}. For sake of completeness, we give an analytical proof of this fact (adapted from \cite{SWY}) in the appendix of this paper. The solution $W_n$
possesses $n$ zeros. The stationary solutions $W_0=\pm 1$ correspond to the zero curvature
solution. Linearizing around a stationary solution $W_n$ leads to the
linear operator 
\begin{equation*}
\cA_n=-\partial_x^2+\frac{(1- \frac{2m}{r}) }{r^2}(3W_n^2-1).
\end{equation*} 
In \cite{BRZ} it was numerically observed for the first stationary solutions that $\cA_n$ has $n$ negative
eigenvalues. In this paper we show analytically that $\cA_n$ has at
least one negative eigenvalue for $n\ge 1$. An abstract result then
shows that this leads to a nonlinear instability. We will describe in
Section \ref{Sec2} a general abstract setting for non linear one
dimensional wave equations. This abstract setting is applied in
Section \ref{Sec3} to the Yang-Mills equation.

\subsection{The exterior of the Schwarzschild black hole} The exterior
Schwarzschild spacetime is given by ${\mathcal M}=\R_t\times
\R_{r>2m}\times S^2$ equipped with the metric 
\beaa
\notag
g &=&  - (1 - \frac{2m}{r})dt^{2} + \frac{1}{ (1 - \frac{2m}{r})} dr^{2} + r^{2} d\th^{2} +  r^{2}\sin^{2} (\th) d\phi^{2} \\
&=& N(-dt^2+d{x}^{2})+r^2d\si^2
\eeaa
where
\bea
N &=& (1 - \frac{2m}{r})
\eea
and $d\si^2$ is the usual volume element on the sphere. The coordinate
$x$ is defined by the requirement
\begin{equation*}
\frac{dx}{dr}=N^{-1}.
\end{equation*}
The coordinates $t,r, \th, \phi$, are called Boyer-Lindquist coordinates. The
singularity $r=2m$ is a coordinate singularity and can be removed by
changing coordinates, see \cite{HE}. $m$ is the mass of the black hole. We will only
be interested in the region outside the black hole, $r>2m$.

\subsection{The spherically symmetric $SU(2)$ Yang-Mills equations on the Schwarz\-schild metric} \label{sphericallysymmetricYM}

Let $G = SU(2)$, the real Lie group of $2 \text{x} 2$ unitary matrices
of determinant 1. The Lie algebra associated to $G$ is $su(2)$, the
antihermitian traceless $2 \text{x} 2$  matrices. Let $\tau_{j}$, $j
\in \{1, 2, 3 \}$, be the following real basis of $su(2)$:
\begin{eqnarray*}
\tau_1=\frac{i}{2}\left(\begin{array}{cc} 0 & 1 \\ 1 &
    0\end{array}\right),\quad
\tau_2=\frac{1}{2}\left(\begin{array}{cc} 0 & -1 \\ 1 &
    0\end{array}\right),\quad
\tau_3=\frac{i}{2}\left(\begin{array}{cc} 1 & 0\\ 0 &
    -1 \end{array}\right). 
\end{eqnarray*}
Note that 
\beaa
[\tau_1,\tau_2]=\tau_3,\quad  [\tau_3,\tau_1]=\tau_2,\quad [\tau_2,\tau_3]=\tau_1.
\eeaa

We are looking for a connection $A$, that is a one form with values in the Lie algebra $su(2)$ associated to the Lie group $SU(2)$, which satisfies the Yang-Mills equations which are:
\bea
\text{\bf D}^{(A)}_{\a} F^{\a\b} \equiv \der_{\a} F^{\a\b} + [A_{\a}, F^{\a\b} ]  = 0, \label{eq:YM}
\eea
where $[.,.]$ is the Lie bracket and $F_{\a\b}$ is the Yang-Mills curvature given by
 \bea
F_{\a\b} = \der_{\a}A_{\b} - \der_{\b}A_{\a} + [A_{\a},A_{\b}], \label{defYMcurvature}
\eea
and where we have used the Einstein raising indices convention with respect to the Schwarzschild metric. We also have the Bianchi identities which are always satisfied in view of the symmetries of the Riemann tensor and the Jacobi identity for the Lie bracket:
\bea
\text{\bf D}^{(A)}_{\a}F_{\mu\nu} + \text{\bf D}^{(A)}_{\mu}F_{\nu\a} + \text{\bf D}^{(A)}_{\nu} F_{\a\mu} = 0. \label{eq:Bianchi}
\eea

The Cauchy problem for the Yang-Mills equations formulates as the following: given a Cauchy hypersurface $\Sigma$ in $M$, and a ${\mathcal G}$-valued one form $A_{\mu}$ on $\Sigma$, and a ${\mathcal G}$-valued one form $E_{\mu}$ on $\Sigma$ satisfying
\begin{eqnarray}
\label{YMconstraintsone}
\left.\begin{array}{rcl} E_{t} &=& 0, \\
\textbf{D}^{(A)}_{\mu}E^{\mu} &=& 0\end{array}\right\} 
\end{eqnarray}
we are looking for a ${\mathcal G}$-valued two form $F_{\mu\nu}$ satisfying the Yang-Mills equations such that once $F_{\mu\nu}$ restricted to $\Sigma$ we have
\bea
F_{\mu t} = E_{\mu}  \label{YMconstraintstwo}
\eea
and such that $F_{\mu\nu}$ corresponds to the curvature derived from the Yang-Mills potential $A_{\mu}$, i.e. given by \eqref{defYMcurvature}. Equations \eqref{YMconstraintsone} are the Yang-Mills constraints equations on the initial data.

Any spherically symmetric Yang-Mills potential can be written in the
following form after applying a gauge transformation, see \cite{FM}, \cite{GuHu} and \cite{W},
\begin{eqnarray}
\label{SPAA}
A &=& [ -W_{1}(t, r) \tau_{1} - W_{2}(t, r) \tau_{2} ] d\th  + [ W_{2}
(t, r) \sin (\th) \tau_{1} - W_{1} (t, r) \sin (\th) \tau_{2}] d\phi\nonumber\\
& + & \cos (\th) \tau_{3} d\phi + A_{0} (t, r) \tau_{3} dt  + A_{1} (t, r)  \tau_{3} dr, 
\end{eqnarray}
where $A_{0} (t, r) $, $A_{1} (t, r) $, $W_{1}(t, r)$, $W_{2}(t, r)$
are arbitrary real functions. We consider here a purely magnetic
Ansatz in which we have $A_0=A_1=W_2=0,\, W_1=:W$. The components of
the curvature are then 
\beaa
\left.\begin{array}{rcl} F_{\th x} &=& W' \tau_{1},\\
F_{\th t}  &=& \dot{W} \tau_{1}, \\
F_{\phi x} &=& W' \sin (\th) \tau_{2}, \\
F_{\phi t} &=&\dot{W} \sin (\th) \tau_{2}, \\
F_{tx} &=& 0,\\
F_{\th\phi} &=&   ( W^{2} -1 ) \sin (\th) \tau_{3}. \end{array}\right\}
\eeaa
This kind of Ansatz is preserved by the evolution. Also the principal
restriction is $A_0=A_1=0$.  The constraint equations then impose that
$W_1$ is proportional to $W_2$, a case which can be reduced to
$W_2=0$. We refer the reader to \cite{GH} for details. 

\subsection{The initial value problem for the purely magnetic Ansatz} \label{AnsatzforinitialdataYM}
We look at initial data prescribed on $t=0$ where there exists a
gauge transformation such that once applied on the initial data, the
potential $A$ can be written in this gauge as
\begin{equation}
\label{Ansatz}
\left.\begin{array}{rcl}
A_{t} (t=0)  &=& 0, \\
A_{r} (t=0)  &=& 0, \\
A_{\th} (t=0)  &=& -W_0(r)\tau_{1},  \\
A_{\phi} (t=0)  &=& -W_0( r)  \sin
(\th) \tau_{2} + \cos (\th) \tau_{3}, \end{array}\right\} 
\end{equation}
and, we are given in this gauge the following one form $E_{\mu}$ on $t=0$:
\begin{equation}
\label{AnsatzE}\left.\begin{array}{rcl}
E_{\th} (t=0) &=& F_{\th t} (0)   = W_1(r) \tau_{1},  \\
E_{\phi} (t=0) &=& F_{\phi t} (0)   = W_1(r) \sin (\th)   \tau_{2}, \\
E_{r} (t=0) &=& F_{rt} (0) = 0, \\
E_{t} (t=0) &=& F_{tt} (t=0) = 0. \end{array}\right\}
\end{equation}
Notice that with this Ansatz the constraint equations \eqref{YMconstraintsone} are
automatically fulfilled
\begin{eqnarray}
\label{constraintintheAnsatz}
\notag
( { \text{\bf D}^{(A)}}^{\th}  E_{\th} + {\text{\bf D}^{(A)}}^{\phi}  E_{\phi} + {\text{\bf D}^{(A)}}^{r}  E_{r} ) (t=0)= 0.
\end{eqnarray}

The Yang-Mills equations now reduce to 
\begin{equation}
\label{YMSW}
\left.\begin{array}{rcl} \ddot{W}-W''+PW(W^2-1)&=&0,\\
W(0)&=&W_0,\\
\partial_t W(0)&=&W_1,\end{array}\right\} 
\end{equation}
where 
\begin{equation*}
P=\frac{(1-\frac{2m}{r})}{r^2}.
\end{equation*}
It is easy to check that the following energy is conserved, see also \cite{GH},
\begin{equation*}
\cE(W,\dot{W})=\int \dot{W}^2+(W')^2+\frac{P}{2}(W^2-1)^2 dx.
\end{equation*}

We note by $\dot{H}^k=\dot{H}^k(\R, dx)$ and
  $H^k=H^k(\R,dx)$, the homogeneous and inhomogeneous Sobolev spaces
of order $k$, respectively. 

\begin{definition}
\begin{enumerate}
\item We define the spaces $L^4_P$, resp. $L^2_P$, as the completion of
$C_0^{\infty}(\R)$ for the norm
\bea
\Vert v\Vert_{L^4_P}^4:=\int P\vert v\vert^4 dx\quad
\mbox{resp.}\quad \Vert v\Vert_{L^2_P}^2:=\int P \vert v\vert^2 dx. 
\eea
\item
We also define for $1\le k\le 2$ the space $\cH^k$ as the completion
of $C_0^{\infty}(\R)$ for the norm 
\bea
\Vert u\Vert^2_{\cH^k}=\Vert u\Vert_{\dot{H}^k}^2+\Vert
u\Vert_{L^4_P}^2. 
\eea
\end{enumerate}
\end{definition}
We note that $\cH^k$ is a Banach space which contains all constant
functions. It turns out that $\cE:=\cH^1\times L^2$ is exactly the space
of finite energy solutions, see \cite{GH} for details. We then have \cite[Theorem 1]{GH}
\begin{theorem}
\label{ThGEYM}
Let $(W_0,W_1)\in \cH^2\times H^1$. Then there exists a unique strong solution of
\eqref{YMSW} with
\begin{eqnarray*}
W&\in&C^1([0,\infty);\cH^1)\cap
C([0,\infty);\cH^2),\\
\partial_tW&\in& C^1([0,\infty);L^2)\cap C([0,\infty);H^1),\\
\sqrt{P}(W^2-1)&\in&C^1([0,\infty);L^2)\cap C([0,\infty);H^1). 
\end{eqnarray*}
\end{theorem}
We can
reformulate the above theorem in the following way
\begin{corollary}
\label{Cor1}
We suppose that the initial data for the Yang-Mills equations is given
after suitable gauge transformation by
\begin{eqnarray*}
\left.\begin{array}{rcl} A_t(0)&=&A_r(0)=0,\\
    A_{\theta}(0)&=&-W_0\tau_1,\\
 A_{\phi}(0)&=&-W_0\sin
    \theta\tau_2+\cos\theta\tau_3,\\
E_{\theta}(0)&=&W_1\tau_1,\\
E_{\phi}(0)&=&W_1\sin\theta\tau_2,\\
E_r(0)&=&E_t(0)=0\end{array}\right\} 
\end{eqnarray*}
with $(W_0,W_1)\in \cH^2\times H^1$. Then, the Yang-Mills equation \eqref{eq:YM} admits a unique
solution $F$ with
\begin{eqnarray*}
F_{\theta x},\, \frac{1}{\sin\theta}F_{\phi x},\, F_{\theta t},\,
\frac{1}{\sin\theta}F_{\phi
  t},\sqrt{P}\frac{1}{\sin\theta}F_{\theta\phi}&\in&
C^1([0,\infty);L^2)\cap C([0,\infty);H^1).
\end{eqnarray*}
\end{corollary}
\subsection{Energies}
We now introduce the Yang-Mills energy momentum tensor 
\begin{equation*}
T_{\mu\nu}=\<F_{\mu\beta},F_{\nu}^{\beta}\>-\frac{1}{4}g_{\mu\nu}\<F_{\alpha\beta},F^{\alpha\beta}\>.
\end{equation*}
Here $\<.,.\>$ is an Ad-invariant scalar product on the Lie algebra $su(2)$.
We have 
\begin{equation*}
\nabla^{\nu}T_{\mu\nu}=0.
\end{equation*}
For a vector field $X^{\nu}$ we define 
\begin{equation*}
J_{\mu}(X)=X^{\nu}T_{\mu\nu}
\end{equation*}
and the energy on the spacelike slice $\Sigma_t$
($\Sigma_{t_0}=\{t=t_0\}$ ) by
\begin{equation*}
E^{(X)}(F(t))=\int_{\Sigma_t}J_{\mu}(X)n^{\nu}d_{\Sigma_t}.
\end{equation*}
By the divergence theorem this energy is conserved if $X$ is
Killing. In particular 
\begin{equation*}
E^{(\p_t)}(F(t))=\int _{\Sigma_t}J_{\mu}(\p_t)n^{\mu}d_{\Sigma_t}
\end{equation*}
is conserved. If $F$ is the curvature associated to $(W,\dot{W})$, then 
\begin{equation*}
E^{(\p_t)}(F(t))=\cE(W,\dot{W}),
\end{equation*}
see \cite{GH} for details. 
\subsection{Main result}
We first recall the following result which is implicit in the paper
\cite{BRZ} of P. Bizo\'n, A. Rostworowski and A. Zenginoglu.
\begin{theorem}
\label{thstat}
There exists a decreasing sequence $\{a_n\}_{n\in \N^{\ge1}},\, 0<...<
a_n< a_{n-1}<...<a_1=\frac{1+\sqrt{3}}{3\sqrt{3}+5}$ and smooth stationary solutions $W_n$ of \eqref{YMSW}
with 
\begin{equation*}
-1\le W_n\le 1,\quad \lim_{x\rightarrow -\infty}W_n(x)=a_n,\quad \lim_{x\rightarrow
  \infty}W_n(x)=(-1)^n. 
\end{equation*}
The solution $W_n$ has exactly $n$ zeros. 
\end{theorem}

\begin{remark}
	There is an explicit formula for the first stationary solution
        (see \cite{BCC})
	$$W_1 = \frac{c-\frac{r}{2m}}{\frac{r}{2m}+3(c-1)}, \quad c=\frac{3+\sqrt{3}}{2}.$$
	This solution corresponds to  $\lim_{x\rightarrow -\infty}W_1(x)=a_1=\frac{1+\sqrt{3}}{3\sqrt{3}+5}$.
\end{remark}

We give a detailed proof of this result in the appendix, where we
follow arguments of Smoller, Wasserman, Yau and McLeod. The above
solutions are all nonlinearly instable :
\begin{theorem}[Main Theorem]
\label{Mainth}
For all $n\ge 1$ the solution $W_n$ of \eqref{YMSW} is unstable. More precisely there
exists $\epsilon_0>0$ and a sequence $(W_{0,n}^m,W_{1,n}^m)$ with $\Vert
(W_{0,n}^m,W_{1,n}^m)-(W_n,0)\Vert_{\cE}\rightarrow 0,\, m\rightarrow \infty$, but for all
$m$
\begin{equation*}
\sup_{t\ge 0}\Vert(W_n^m(t),\partial_tW_n^m(t))-(W_n,0)\Vert_{\cE}\ge \epsilon_0>0. 
\end{equation*} 
\end{theorem}

\begin{remark}
	We don't show in this paper that there is no stationary solution with $W(2m)>a_1$. We do not exclude either the fact that there may exist solutions with an infinite number of zeros which tend to zero at infinity. Our main theorem does not apply to this two categories of hypothetical stationary solutions.
\end{remark}

For $n$ given we construct initial data from $W_n$ as in Section
\ref{AnsatzforinitialdataYM}. Let $F_{n}$ be the corresponding
curvature at time $t=0$. We obtain
\begin{corollary}
\label{corstat}
For all $n\ge 1$ the solution $F_n$ of \eqref{eq:YM} is
unstable. More precisely there exists $\epsilon_0>0$ and a sequence of
initial data giving rise to the curvature $F_{0,n}^m$ with 
\begin{equation*}
E^{(\p_t)}(F_{0,n}^m-F_{n})\rightarrow 0,\quad m\rightarrow \infty,
\end{equation*}
but for all $m$
\begin{equation*}
\sup_{t\ge 0}E^{(\p_t)}(F_n^m(t)-F_{n})\ge \epsilon_0,
\end{equation*}
where $F_n^m(t)$ is the solution associated to the initial data
corresponding to the curvature $F_{0,n}^m$. 
\end{corollary}
\textbf{Acknowledgments.} The first author acknowledges support from
the ANR funding ANR-12-BS01-012-01. Both authors thank Sari Ghanem for
fruitful discussions on Yang-Mills equations. 
\section{Abstract setting}
\label{Sec2}
\subsection{Abstract result}
We consider the one dimensional wave equation
\begin{equation}
\label{AWE}
\left\{\begin{array}{rcl} \ddot{u}-u''+Vu&=&F(u),\\
u\vert_{t=0}&=&u_0,\\
\partial_tu\vert_{t=0}&=&u_1 \end{array}\right.
\end{equation}
with $\dot{}=\partial_t,\, '=\partial_x$ and
\begin{equation}
\tag{HV}\label{HV}
V\in C(\R)\cap L^1(\R),\, \lim_{\vert x\vert\rightarrow \infty}V(x)=0,\, \int_{\R}V(x)dx<0.
\end{equation}
We also suppose that 
\begin{equation}
\tag{HF}\label{HF}
\Vert F(u)-F(v)\Vert_{L^2}\le M_F(\Vert u\Vert_{H^1}+\Vert
v\Vert_{H^1})\Vert u-v\Vert_{H^1}
\end{equation}
for $\Vert u\Vert_{H^1}\le 1,\Vert v\Vert_{H^1}\le 1$. Let
$X=H^1\times L^2$. We then have the following 
\begin{theorem}
The zero solution of \eqref{AWE} is unstable. More precisely there
exists $\epsilon_0>0$ and a sequence $(u_0^m,u_1^m)$ with $\Vert
(u_0^m,u_1^m)\Vert_X\rightarrow 0,\, m\rightarrow \infty$, but for all
$m$
\begin{equation*}
\sup_{t\ge 0}\Vert(u^m(t),\partial_tu^m(t))\Vert_X\ge \epsilon_0>0. 
\end{equation*} 
Here $u^m(t)$ is the solution of \eqref{AWE} with initial data
$(u_0^m,u_1^m)$ and the supremum is taken over the maximal interval of
existence of $u^m(t)$. 
\end{theorem}
Let 
\begin{equation*}
\cA=-\partial_x^2+V,\quad D(\cA)=H^2(\R).
\end{equation*}
We note that $\cA$ is a selfadjoint operator. 
\subsection{Spectral analysis of $\cA$}
\begin{proposition}
We have 
\begin{equation*}
\sigma(\cA)=\{-\lambda_n^2\}_{n\in \cN}\cup [0,\infty),
\end{equation*}
where $-\lambda_n^2,\quad \lambda_0> \lambda_1>....\lambda_n> ...>0$ is a finite ($\cN=\{0,...,N\}$) or infinite $(\cN=\N)$ sequence of negative eigenvalues with only possible
accumulation point $0$. 
\end{proposition}
\proof 

First note that $\sigma(\cA)\cap \R^-\neq \emptyset$. Indeed let $\chi\in C_0^{\infty}(\R),\, \chi(0)=1,\, \chi\ge 0,\,
\chi_R(.)=\chi(\frac{.}{R})$. Then 
\begin{equation*}
\<\cA\chi_R,\chi_R\>=\frac{1}{R}\int \vert \chi'(x)\vert^2dx+\int
V(x)\chi_R^2dx\rightarrow \int V(x) dx<0,\quad R\rightarrow
\infty. 
\end{equation*}
We now introduce the comparison operator 
\begin{equation*}
\cB=-\partial_x^2. 
\end{equation*}
We compute
\begin{equation*}
(\cB-z^2)^{-1}-(\cA-z^2)^{-1}=(\cA-z^2)^{-1}V(\cB-z^2)^{-1}.
\end{equation*}
Using that $\lim_{x\rightarrow \pm \infty} V(x)=0$ we see that this is
a compact operator. By the Weyl criterion
\begin{equation*}
\sigma_{ess}(\cA)=\sigma_{ess}(\cB)=[0,\infty). 
\end{equation*}
On the other hand we already know that $\cA$ has negative
spectrum. It therefore has at least one negative eigenvalue. $\cA$
being bounded from below the proposition follows. 
\qed 

\subsection{The wave equation as a first order equation}
\subsubsection{The linear equation}
The equation
\begin{equation*}
\ddot{v}+\cA v=0
\end{equation*}
is equivalent to 
\begin{equation*}
\partial_t\psi=L\psi,\quad L=\left(\begin{array}{cc} 0 & i \\
    i\cA & 0 \end{array}\right),\quad \psi=\left(\begin{array}{c} v
    \\ \frac{1}{i} \partial_t v \end{array}\right).  
\end{equation*}
\begin{remark}
Let 
\begin{equation*}
\cA\phi_0=-\lambda^2\phi_0.
\end{equation*}
Then we have 
\begin{enumerate}
\item $\phi_0\in H^2$. 
\item Let $\psi^{\pm}_0=\left(\begin{array}{c} \phi_0 \\ \pm\frac{1}{i} \lambda
      \phi_0\end{array}\right)$. Then 
\begin{equation*}
L\psi^{\pm}_0=\pm\lambda\psi^{\pm}_0.
\end{equation*}
\end{enumerate}
\end{remark}
Let $V_-$ be the negative part of the potential. For $\mu^2>\Vert
V_-\Vert_{\infty}(\ge \lambda_0^2)$ we introduce the scalar product 
\begin{equation*}
\< u,v\>_{\mu}=\<(\cA+\mu^2)u_0,v_0\>+\<u_1,v_1\>
\end{equation*}
where $\<.,.\>$ is the usual scalar product on $\cH=L^2(\R)$. We note
$\Vert.\Vert_{\mu}$ the corresponding norm. It is easy to check that
the norms $\Vert .\Vert_{\mu}$ and
$\Vert.\Vert_X$ are equivalent. 
\begin{proposition}
$L$ is the generator of a $C^0-$ semigroup $e^{tL}$ on $X$.
\end{proposition}
\proof 

Let $\mu^2>\Vert V_-\Vert$ and
\begin{equation*}
L_{\mu}=\left(\begin{array}{cc} 0 & i\\ i(\cA+\mu^2) &
    0 \end{array}\right),\quad B_{\mu}=\left(\begin{array}{cc} 0 & 0\\ -i\mu^2 &
    0 \end{array}\right).
\end{equation*}
$iL_{\mu}$ is a selfadjoint operator on $(X, \<.,.\>_{\mu})$ and in particular the
generator of a $C^0-$ semigroup $e^{L_{\mu}t}$. We have
$L=L_{\mu}+B_{\mu}$. $B_{\mu}$ being bounded, we can apply
\cite[Theorem 3.1.1]{Pa} to see that $L$ is the generator of
a $C^0-$ semigroup on $(X, \Vert .\Vert_{\mu})$ and thus on $(X,\Vert .\Vert_X)$. 
\qed

Let now 
\begin{equation*}
M_i=\left(\begin{array}{cc} \one & \one \\
    \frac{\lambda_i}{i} & -\frac{\lambda_i}{i} \end{array} \right). 
\end{equation*}

Note that $det M_i=2i\lambda_i\neq 0$ and that $M_i$ is thus invertible. We define
$P_i=\one_{\{-\lambda_i^2\}}(\cA)M_i$ and $X_i=P_iX$. We also define
$X_{\infty}=\one_{\R^+}(\cA)\one_2X.$  Here
$\one_{\{-\lambda_i^2\}}(\cA)$ and $\one_{\R^+}(\cA)$ are defined by
the spectral theorem. In particular $\one_{\{-\lambda_i^2\}}(\cA)$ is
the projection on the eigenspace of $\cA$ associated to the eigenvalue
$-\lambda_i^2$. 
\begin{lemma}
\begin{equation*}
X=\left(\oplus_{i\in \cN}X_i\right)\oplus X_{\infty}.
\end{equation*}
\end{lemma}
\begin{remark}
Note that the sum is orthogonal with respect to the scalar product
$\<.,.\>_{\mu}$. 
\end{remark}
\proof 

Let $(\phi,\psi)\in X$. We put
\begin{eqnarray*}
\phi_i&=&\one_{\{-\lambda_i^2\}}(\cA)\phi,\quad
\psi_i=\one_{\{-\lambda_i^2\}}(\cA)\psi,\quad
\left(\begin{array}{c} \tilde{\phi}_i \\ \tilde{\psi}_i\end{array}\right)=M_i^{-1}\left(\begin{array}{c} \phi_i \\ \psi_i\end{array}\right).
\end{eqnarray*}
Since $\cA$ is self-adjoint, we can write
$$\phi = \sum_{i\in \cN} \phi_i +  \one_{\R^+}(\cA)\phi, \quad \psi = \sum_{i\in \cN} \psi_i +  \one_{\R^+}(\cA)\psi.$$
Then 
\begin{equation*}
\left(\begin{array}{c} \phi \\
    \psi\end{array}\right)=\sum_{i\in \cN}M_i\left(\begin{array}{c}
    \tilde{\phi}_i \\
    \tilde{\psi}_i\end{array}\right)+\left(\begin{array}{c}
    \one_{\R^+}(\cA)\phi \\ \one_{\R^+}(\cA)\psi \end{array}\right)
\end{equation*}
gives the required decomposition.  For uniqueness let 
\begin{equation*}
\left(\begin{array}{c} \phi_i \\
    \psi_i\end{array}\right)=\sum_{i\in\cN}M_i\left(\begin{array}{c}
    \tilde{\phi}_i \\
    \tilde{\psi}_i\end{array}\right)+\left(\begin{array}{c}
    \phi_{\infty} \\ \psi_{\infty} \end{array}\right)
\end{equation*}
Applying $\one_{\R^+}(\cA),\, \one_{\{-\lambda_i^2\}}(\cA)$ to each line immediately gives 
\begin{equation*}
\phi_{\infty}=\one_{\R^+}(\cA)\phi,\quad
\psi_{\infty}=\one_{\R^+}(\cA)\psi,\quad
\left(\begin{array}{c} \tilde{\phi}_i \\ \tilde{\psi}_i\end{array}\right)=M_i^{-1} \left(\begin{array}{c} \phi_i \\ \psi_i\end{array}\right),
\end{equation*}
where $\psi_i=\one_{\{-\lambda_i^2\}}(\cA)\psi,\,
\phi_i=\one_{\{-\lambda_i^2\}}(\cA)\phi$.
\qed

Let 
\begin{equation*}
X_i^{\pm}=M_i\one_{\{-\lambda_i^2\}}(\cA)P_{\pm}X,
\end{equation*}
where $P_+(\phi,\psi)=(\phi,0),\, P_-(\phi,\psi)=(0,\psi)$. Clearly
$X_i=X_i^+\oplus X_i^-$ and thus
\begin{equation*}
X=\left(\bigoplus_{i\in \cN}(X_i^+\oplus X_i^-)\right)\oplus X_{\infty}.
\end{equation*}

\begin{remark}
	Let $(\phi_i,\psi_i) \in X_i^{\pm}$. Then $L (\phi_i,\psi_i)= \pm \lambda_i (\phi_i,\psi_i)$.
\end{remark}

\begin{remark}
On $X_i$ the norm $\Vert.\Vert_{{\sqrt{2}\lambda_i}}$ is equivalent to
the norm $\Vert.\Vert_X$ and $X_i^+,\, X_i^-$ are orthogonal
with respect to this scalar product. Indeed :
\end{remark}
\begin{equation*}
\left\<\left(\begin{array}{c} \phi \\
    \frac{\lambda_i}{i}\phi \end{array}\right),\left(\begin{array}{c}
    \psi \\
    -\frac{\lambda_i}{i}\psi \end{array}\right)\right\>_{\sqrt{2}\lambda_i}=\lambda_i^2\<\phi,\psi\>-\lambda_i^2\<\phi,\psi\>=0. 
\end{equation*}

\begin{proposition}
\label{prop3}
\begin{enumerate}
\item The spaces $X_i,\, X_{\infty}$ are $e^{tL}$ invariant. 
\item For all $\epsilon>0$ there exists $C_{\epsilon}>0$ such that for all $i\in \cN$ and for
  all $t\in \R$
\begin{equation*}
\Vert e^{tL}\vert_{X_i}\Vert_{X\rightarrow X}\le C_{\epsilon} e^{(\lambda_i+\epsilon)\vert t\vert}.
\end{equation*}
\item For all $\epsilon>0$ there exists $C_{\epsilon}>0$ such that for
  all $t\in \R$
\begin{equation*}
\Vert e^{tL}\vert_{X_{\infty}}\Vert_{X\rightarrow X} \le C_{\epsilon} e^{\epsilon \vert
  t\vert}. 
\end{equation*}
\end{enumerate}
\end{proposition}
\proof 

(1) We have 
\begin{equation*}
e^{tL}M_i\one_{\{-\lambda_i^2\}}(\cA)\one_2\left(\begin{array}{c}
    \phi \\ \psi\end{array}\right)=M_i\one_{\{-\lambda_i^2\}}(\cA)\one_2\left(\begin{array}{c}
    e^{t\lambda_i} \phi \\ e^{-t\lambda_i}\psi \end{array}\right)
\end{equation*}
and thus $X_i$ is invariant under the evolution. 
The fact that $X_{\infty}$ is invariant follows from the
fact that $\one_{\R^+}(\cA)$ commutes with $L$. 

(2) Because of the equivalence of the norms it is sufficient to
estimate the $\Vert.\Vert_{\mu}$ norm. Let 
\begin{equation*}
\left(\begin{array}{c} \phi_i\\ \psi_i\end{array}\right)\in X_i.
\end{equation*}
We compute 

\begin{eqnarray*}
\left\Vert e^{tL}\left(\begin{array}{c} \phi_i \\ \psi_i\end{array}\right)\right\Vert_{\mu}^2
&=&\left\Vert \left(\begin{array}{cc} (\mu^2-\lambda_i^2)^{1/2}
        & 0 \\ 0 & \one \end{array}\right)M_i\left(\begin{array}{cc}
        e^{t\lambda_i} & 0 \\ 0 & e^{-\lambda_i
          t} \end{array}\right)M_i^{-1}\left(\begin{array}{c}
        \phi_i \\ \psi_i\end{array}\right)\right\Vert_{\cH\times
    \cH}\\
&\le&\Vert N_i\Vert^2_{\R^2\rightarrow \R^2}\left\Vert \left(\begin{array}{c}
        \phi_i \\ \psi_i\end{array}\right)\right\Vert^2_{\mu},
\end{eqnarray*}
where 
\begin{equation*}
N_i=\left(\begin{array}{cc} (\mu^2-\lambda_i^2)^{1/2}
        & 0 \\ 0 & \one \end{array}\right)M_i\left(\begin{array}{cc}
        e^{t\lambda_i} & 0 \\ 0 & e^{-\lambda_i
          t} \end{array}\right)M_i^{-1}\left(\begin{array}{cc} (\mu^2-\lambda_i^2)^{-1/2}
        & 0 \\ 0 &  \one \end{array}\right).
\end{equation*}
We then estimate uniformly in $i\in \cN$:
\begin{eqnarray*}
\Vert N_i\Vert^2_{\R^2\rightarrow \R^2}&\lesssim& \left\Vert
  \frac{1}{2}\left(\begin{array}{cc} e^{t\lambda_i}+e^{-t\lambda_i} & \frac{1}{i\lambda_i}(e^{-t\lambda_i}-e^{t\lambda_i}) \\
      \frac{\lambda_i}{i}(e^{t\lambda_i }-e^{-t\lambda_i}) &
      e^{t\lambda_i}+e^{-\lambda_i t} \end{array}\right)\right\Vert_2^2.\\
\end{eqnarray*}
We have for $t\ge 0$
\begin{eqnarray*}
\frac{1}{\lambda_i}(e^{t\lambda_i}-e^{-t\lambda_i})&=&2\sum_{i=1}^{\infty}\frac{(t\lambda_i)^{2i+1}}{\lambda_i(2i+1)!}\\
&\le&2t\sum_{i=1}^{\infty}\frac{(t\lambda_i)^{2i}}{(2i)!}\le
t(e^{t\lambda_i }+e^{-t\lambda_i})\le \tilde{C}_{\epsilon} e^{(\lambda_i+\epsilon) t}. 
\end{eqnarray*}
Using that $\lambda_i\le \lambda_0$ we find uniformly in $i\in \cN$:
\begin{equation*}
\Vert N_i\Vert_{\R^2\rightarrow \R^2}\lesssim
e^{(\lambda_i+\epsilon)\vert t\vert}.
\end{equation*}
(3) We consider the case $t\ge 0$. First note that 
\begin{equation*}
\Vert u\Vert_{X_\epsilon}^2=\<\cA u_0,u_0\>+\Vert u_1\Vert^2+\epsilon^2\Vert u_0\Vert^2
\end{equation*}
defines a norm on $X_{\infty}$. We estimate for $u(t)=e^{tL}u$ 
\begin{eqnarray*}
\frac{d}{dt}\Vert u\Vert_{X_\epsilon}^2&=&2{\rm
  Re}\left(\<\cA u_0,\dot{u}_0\>+\<u_1,\dot{u}_1\>+\epsilon^2\<u_0,\dot{u}_0\>\right)\\
&=&2{\rm Re}\, \epsilon^2\<u_0,i u_1\>\\
&\le&2\epsilon^2\Vert u_0\Vert\Vert u_1\Vert\le \epsilon^3\Vert
u_0\Vert^2+\epsilon\Vert u_1\Vert^2\le \epsilon \Vert u\Vert_{X_\epsilon}^2.
\end{eqnarray*}
By the Gronwall lemma we obtain:
\begin{equation*}
\Vert u(t)\Vert_{X_\epsilon}^2\le \tilde{C}_{\epsilon} e^{\epsilon t}\Vert
u\Vert^2_{X_\epsilon}. 
\end{equation*}
We now claim that on $X_{\infty}$ the $X$ and the $X_\epsilon$ norms are
equivalent. Indeed
\begin{eqnarray*}
\<\cA u_0,u_0\>+\Vert u_1\Vert^2+\epsilon^2\Vert u_0\Vert^2\lesssim \Vert u_0\Vert_{H^1}^2+\Vert u_1\Vert^2.
\end{eqnarray*}
Also,
\begin{eqnarray*}
\Vert u_0\Vert_{H^1}^2+\Vert
u_1\Vert^2&=&\<(-\partial_x^2+V)u_0,u_0\>-\<Vu_0,u_0\>+\Vert
u_0\Vert^2+\Vert u_1\Vert^2\\
&\lesssim& \<\cA u_0,u_0\>+\Vert
u_0\Vert^2+\Vert u_1\Vert^2\lesssim \Vert u\Vert_{X_\epsilon}^2.  
\end{eqnarray*}
Then we can estimate 
\begin{eqnarray*}
\Vert u(t)\Vert_X\lesssim \Vert u(t)\Vert_{X_\epsilon}\lesssim e^{\epsilon
  t}\Vert u\Vert_{X_\epsilon}\lesssim e^{\epsilon t}\Vert u\Vert_X. 
\end{eqnarray*}
\qed

Let $Y=X_0^-\oplus\left(\bigoplus_{i=1}^{N}X_i\right)\oplus X_{\infty}.$ We have $X=X^+_0\oplus Y$ and both
spaces are invariant under $e^{tL}$. 
\begin{corollary}
\label{cor1}
For all $\epsilon>0$ there exists $M_{L,\epsilon}>0$ such that for all $t\ge 0$ we have 
\begin{equation*}
\Vert e^{tL}\vert_{Y}\Vert_{X\rightarrow X}\le M_{L,\epsilon} e^{(\lambda_1+\epsilon) t}.
\end{equation*}
\end{corollary}
\proof 

Because of the equivalence of the norms $\Vert.\Vert_{X}$ and
$\Vert.\Vert_{\mu}\, (\mu^2>\Vert V_-\Vert_{\infty})$ it is
sufficient to show the estimate with respect to the norm $\Vert.\Vert_{\mu}$. We
choose $\epsilon<\lambda_1$ and apply Proposition \ref{prop3}. Let
\begin{equation*}
\phi=\phi_0^-+\sum_{i=1}^{N} \phi_i+\phi_{\infty}
\end{equation*}
with $\phi_0^-\in X_0^-,\, \phi_i\in X_i,\, \phi_{\infty}\in X_{\infty}$. We
have 
\begin{eqnarray*}
\Vert e^{tL}\phi\Vert_{\mu}^2&=&e^{-\lambda_0 t}\Vert \phi_0^-\Vert^2_{\mu}+\sum_{i=1}^{N}\Vert
e^{tL}\phi_i\Vert^2_{\mu}+\Vert
\phi_{\infty}\Vert^2_{\mu}\\
&\lesssim& e^{2(\lambda_1+\epsilon) t }(\Vert
\phi_0^-\Vert^2_{\mu}+\sum_{i=0}^{N} \Vert
\phi_i\Vert^2_{\mu}+\Vert
\phi_{\infty}\Vert^2_{\mu})=e^{2(\lambda_1+\epsilon) t}\Vert
  \phi\Vert^2_{\mu}. 
\end{eqnarray*} 
\qed

Let
\begin{equation*}
E_0=\one_{\{-\lambda_0^2\}}(\cA)M_0P_+M_0^{-1},\, E_1=\one-E_0.
\end{equation*}
We easily check that 
\begin{equation*}
\forall \psi \in X,\, E_0\psi\in X_0^+;\quad \forall \psi\in X,\,
E_1\psi\in Y;\quad E_0+E_1=\one. 
\end{equation*}
\subsubsection{The nonlinear equation}
The nonlinear equation writes now as a first order equation 
\begin{equation}
\label{abstrequ}
\left\{\begin{array}{rcl} \partial_t\psi&=&L\psi+G(\psi),\\
\psi(0)&=&\psi_0 \end{array}\right.
\end{equation}
with 
\begin{equation*}
G(\psi)=\left(\begin{array}{c} 0 \\ F(P_{+}(\psi)) \end{array}\right). 
\end{equation*}
From hypothesis \eqref{HF} we directly obtain
\begin{equation}
\label{LipschitzG}
\Vert G(\psi)-G(\phi)\Vert_{X}\le M_F (\Vert \psi\Vert_X+\Vert
\phi\Vert_X)\Vert \psi-\phi\Vert_X
\end{equation}
for $\Vert \psi\Vert_X\le 1,\, \Vert \phi\Vert_X\le 1$.  The
abstract theorem then writes 
\begin{theorem}
The zero solution of \eqref{abstrequ} is unstable. More precisely
there exists $\epsilon_0>0$ and a sequence $\psi_0^m$ with $\Vert
\psi_0^m\Vert_X\rightarrow 0,\, m\rightarrow \infty$, but for all $m$ 
\begin{equation*}
\sup_{t\ge 0}\Vert\psi^m(t)\Vert_X\ge \epsilon_0>0.
\end{equation*}
Here $\psi^m(t)$ is the solution of \eqref{abstrequ} with initial data
$\psi_0^m$ and the supremum is taken over the maximal interval of
existence of $\psi^m$. 
\end{theorem}

\subsection{Proof of the abstract theorem}

We note $L_0$ the restriction of $L$ to $X^+_0$ and $L_1$ the
restriction of $L$ to $Y$. For $\psi_0\in X^+_0$ with small norm we consider for a certain parameter $\tau>0$ the integral equation 
\begin{equation}
\label{intequ}
\psi(t)=e^{L_0(t-\tau)}\psi_0+\int_{\tau}^te^{L_0(t-s)}E_0G(\psi)ds+\int_{-\infty}^te^{L_1(t-s)}E_1G(\psi)ds=:\I(\psi).
\end{equation}
We fix $\epsilon>0$ in Corollary \ref{cor1} small enough such that
$\tilde{\lambda}_1:=\lambda_1+\epsilon<\lambda_0$. We will drop in the
following the index $\epsilon$ ($M_L=M_{L,\epsilon}$). We fix $\beta>0$ such that $\lambda_0>2\beta>\tilde{\lambda}_1$. Let 
\begin{equation*}
Z=\{\psi\in C([0,\tau];X);\, \Vert\psi\Vert_X\le e^{\beta(t-\tau)}\rho\}.
\end{equation*}
We equip $Z$ with the norm
\begin{equation*}
\Vert \psi\Vert_Z=\sup_{0\le t\le\tau}\Vert e^{-\beta(t-\tau)}\psi(t)\Vert_X.
\end{equation*}
Let $\psi_0$ such that $\Vert\psi_0\Vert_X=\frac{\rho}{3}$. We
claim that for $\rho$ small enough 
\begin{equation*}
\I:\bar{B}_Z(0,\rho)\rightarrow \bar{B}_Z(0,\rho)
\end{equation*} 
and that it is a contraction on that space. First note that 
\begin{equation*}
\I(\psi)=\I_0(\psi)+\I_1(\psi)+\I_2(\psi)
\end{equation*}
with
\begin{eqnarray*}
\I_0(\psi)&=&e^{L_0(t-\tau)}\psi_0,\\
\I_1(\psi)&=&-\int_{t}^{\tau}e^{L_0(u-\tau)}E_0G(\psi(t+\tau-u)) du,\\
\I_2(\psi)&=&\int_{-\infty}^te^{L_1(t-s)}E_1G(\psi(s))ds. 
\end{eqnarray*}
We first estimate for $t\leq \tau$
\begin{equation*}
\Vert \I_0(\psi)\Vert_X= e^{\lambda_0(t-\tau)}\Vert \psi_0\Vert_X\le
1/3 e^{\beta(t-\tau)}\rho. 
\end{equation*}
We then estimate for $\psi\in \bar{B}_Z(0,\rho)$
\begin{eqnarray*}
\Vert\I_1(\psi)\Vert_X&\le&M_F\Vert
E_0\Vert\int_t^{\tau}e^{\lambda_0(u-\tau)}\Vert\psi\Vert_X^2(t+\tau-u)du\\
&\le&M_F\Vert
E_0\Vert\int_t^{\tau}e^{\lambda_0(u-\tau)}\rho^2e^{2\beta(t-u)}du\\
&\le&M_F\Vert E_0\Vert e^{2\beta
  t}e^{-\lambda_0\tau}\rho^2\int_t^{\tau}e^{(\lambda_0-2\beta)u}du\\
&\le&M_F\Vert E_0\Vert\rho^2e^{2\beta
  t}e^{-\lambda_0\tau}\frac{1}{\lambda_0-2\beta}e^{(\lambda_0-2\beta)\tau}\\
&=&\frac{M_F\Vert E_0\Vert\rho^2}{\lambda_0-2\beta}
e^{2\beta(t-\tau)}\\
&\le&\frac{M_F\Vert E_0\Vert\rho^2}{\lambda_0-2\beta}
e^{\beta(t-\tau)}\le 1/3\rho e^{\beta(t-\tau)}
\end{eqnarray*}
for $\rho$ small enough.  We then estimate for $\psi\in
\bar{B}_Z(0,\rho)$ :
\begin{eqnarray*}
\Vert \I_2(\psi(t))\Vert_X&\le&M_LM_F\Vert
E_1\Vert\int_{-\infty}^te^{\tilde{\lambda}_1(t-s)}\rho^2e^{2\beta(s-\tau)}ds\\
&\le&\frac{M_LM_F\Vert E_1\Vert\rho^2}{2\beta-\tilde{\lambda}_1}e^{\tilde{\lambda}_1
  t}e^{-2\beta\tau}e^{(2\beta-\tilde{\lambda}_1)t}\\
&=&\frac{M_LM_F\Vert
  E_1\Vert\rho^2}{2\beta-\tilde{\lambda}_1}e^{2\beta(t-\tau)}\le 1/3\rho^{\beta(t-\tau)}
\end{eqnarray*}
for $\rho$ small enough. We have just proven $\I (\psi) \in \bar{B}_Z(0,\rho)$.
Let us now show that $\I$ is a contraction. We estimate 
\begin{eqnarray*}
\Vert \I_1(\psi)-\I_1(\phi)\Vert_X&\le&2M_F\Vert
E_0\Vert\int_t^{\tau}e^{\lambda_0(u-\tau)}\rho e^{\beta(t-u)}\Vert\psi-\phi\Vert_X(t+\tau-u)du\\
&\le&2M_F\Vert
E_0\Vert\rho\Vert\psi-\phi\Vert_Z\int_t^{\tau}e^{\lambda_0(u-\tau)}e^{2\beta(t-u)}du\\
&=&2M_F\Vert E_0\Vert\rho\Vert\psi-\phi\Vert_Ze^{2\beta
  t}e^{-\lambda_0\tau}\int_t^{\tau}e^{(\lambda_0-2\beta)u}du\\
&\le&\frac{2M_F\Vert
 E_0\Vert\rho}{\lambda_0-2\beta}e^{2\beta(t-\tau)}\le 1/4 e^{\beta(t-\tau)}
\end{eqnarray*}
for $\rho$ sufficiently small. We then estimate 
\begin{eqnarray*}
\Vert \I_2(\psi)-\I_2(\phi)\Vert_X&\le& \int_{-\infty}^t2M_LM_F\Vert
E_1\Vert \rho e^{\tilde{\lambda}_1(t-s)}e^{\beta(s-\tau)}\Vert \psi-\phi\Vert_X ds\\
&\le&2M_LM_F\Vert E_1\Vert\rho\Vert
\psi-\phi\Vert_Z\int_{-\infty}^te^{\tilde{\lambda}_1(t-s)}e^{2\beta(s-\tau)}ds\\
&=&2M_LM_F\Vert E_1\Vert\rho\Vert
\psi-\phi\Vert_Ze^{\tilde{\lambda}_1t}e^{-2\beta
  \tau}\int_{-\infty}^te^{(2\beta-\tilde{\lambda}_1)s}ds\\
&\le&\frac{2M_LM_F\Vert
  E_1\Vert\rho}{2\beta-\tilde{\lambda}_1}e^{2\beta(t-\tau)}\le 1/4e^{\beta(t-\tau)}
\end{eqnarray*}
for $\rho$ sufficiently small. 

It follows that for $\rho$ sufficiently
small there exists a solution of \eqref{intequ} in
$\bar{B}_Z(0,\rho)$. We note this solution $\psi(t,\tau)$. We easily
check that $\psi(t,\tau)$ is also solution of \eqref{abstrequ} with
initial data satisfying
\begin{equation*}
\Vert \psi(0,\tau )\Vert_X\le \rho e^{-\beta\tau}\rightarrow 0,\tau
\rightarrow \infty. 
\end{equation*}
We also estimate 
\begin{eqnarray*}
\Vert \psi(\tau)\Vert_X&\ge&\Vert \psi_0\Vert_X-M_LM_F\Vert
E_1\Vert\int_{-\infty}^{\tau}e^{\tilde{\lambda}_1(\tau-s)}\rho^2e^{2\beta(s-\tau)}ds\\
&=&\rho/3-M_LM_F\Vert
E_1\Vert\rho^2e^{(\tilde{\lambda}_1-2\beta)\tau}\int_{-\infty}^{\tau}e^{(2\beta-\tilde{\lambda}_1)s}ds\\
&\ge&\rho/3-\frac{M_LM_F\Vert E_1\Vert\rho^2}{2\beta-\tilde{\lambda}_1}\ge\rho/6
\end{eqnarray*}
for $\rho$ small enough. It follows that $\psi^m(t)=\psi(t,m)$ does
the job. 
\qed
\begin{remark}
The theorem is close to \cite[Theorem VII.2.3]{DaKr} and \cite[Theorem 5.1.3]{He}.
However both theorems do not apply directly. Whereas \cite[Theorem
VII.2.3]{DaKr} is restricted to bounded operators, \cite[Theorem
5.1.3]{He} only applies if the linear part is sectorial.
\end{remark}
\section{Application of the abstract result to the Yang-Mills
  equation}
\label{Sec3}
First note that if $W(t,r)$ is solution of the Yang-Mills equation
\eqref{YMSW} (written in the $r$ variable), then $ W(2m t,
  2m r)$ is solution of the
same equation with $m=1/2$ and vice versa. We can therefore suppose in
the following $m=1/2$. 
We linearize around $W=W_n$ and obtain for $v=W-W_n$:
\begin{equation*}
\ddot{v}-v''+P(3W_n^2-1)v+Pv^2(v+3W_n)=0. 
\end{equation*}
The linear operator 
\begin{equation*}
\cA_n=-\partial_x^2+P(3W_n^2-1)
\end{equation*}
depends on the stationary solution which we don't
know explicitly. We put 
\begin{equation*}
V_n=P(3W_n^2-1).
\end{equation*}
We first want to apply our abstract result on $X=H^1\times L^2$. It is
easy to see that the nonlinear part fulfills the hypotheses of the
abstract theorem. Indeed we have 
\begin{proposition}
\label{propnonllip}
We have for $\Vert v\Vert_{H^1}\le 1,\, \Vert u\Vert_{H^1}\le 1$:
\begin{equation*}
\Vert F(v)-F(u)\Vert_{L^2}\lesssim(\Vert v\Vert_{H^1}+\Vert u\Vert_{H^1})\Vert u-v\Vert_{H^1}.
\end{equation*}
\end{proposition}
\proof

We compute
\begin{equation*}
F(v)-F(u)=P(v^2+u^2+uv+3(W_nv+W_nu))(u-v).
\end{equation*}
Thus
\begin{eqnarray*}
\Vert F(v)-F(u)\Vert_{L^2}&\lesssim&(\Vert v^2\Vert_{L^2}+\Vert
u^2\Vert_{L^2})\Vert
u-v\Vert_{L^{\infty}}+(\Vert v\Vert_{L^{\infty}}+\Vert v\Vert_{L^{\infty}}\Vert
u\Vert_{L^{\infty}}+\Vert
u\Vert_{L^{\infty}})\Vert u-v\Vert_{L^2}\\
&\lesssim&(\Vert v\Vert^2_{L^4}+\Vert u\Vert_{L^4}^2)\Vert u-v\Vert_{H^1}+(\Vert
v\Vert_{H^1}+\Vert
v\Vert_{H^1}\Vert u\Vert_{H^1}+\Vert u\Vert_{H^1})\Vert v-u\Vert_{H^1}\\
&\lesssim&(\Vert v\Vert^2_{H^1}+\Vert
u\Vert_{H^1}^2+\Vert v\Vert_{H^1}+\Vert u\Vert_{H^1})\Vert
u-v\Vert_{H^1}\\
&\lesssim& (\Vert v\Vert_{H^1}+\Vert u\Vert_{H^1})\Vert
u-v\Vert_{H^1}
\end{eqnarray*}
for $\vert u\Vert_{H^1}\le 1,\, \Vert v\Vert_{H^1}\le 1$. Here we have used the Gagliardo Nirenberg inequality and the Sobolev
embedding $H^1\hookrightarrow L^{\infty}$. 
\qed

In the next subsection we will show that 
\begin{equation*}
\int _{\R}V_n(x)dx<0.
\end{equation*}
\subsection{Study of the potential $V_n$}
Going back to the $r$ variable we see that the potential $W_n$
fulfills the following equation 
\begin{equation}
\label{ym}\left(1-\frac{1}{r}\right)\p_r^2W_n+ \frac{1}{r^2}\p_rW_n+ \frac{1}{r^2}W_n(1-W_n^2)=0
\end{equation}	
with initial data (or boundary condition) $W_n(1)=a_n$, for $0<a_n\leq \frac{1+\sqrt{3}}{5+3\sqrt{3}}$. We also have $\lim_{r\rightarrow \infty}W_n(r)=(-1)^n$. We will drop the index $n$ in the rest of this subsection. 

\subsubsection{A bound on $W$}

\begin{lemma}
We have $-a\leq W\leq a$ for $1\leq r\leq 3$.
\end{lemma}

\begin{proof}
Since the initial data for $W$ are $W(1)=a$ and $W'(1)=-a(1-a^2)<0$, there exists $r_0>1$ such that for $1\leq r \leq r_0$ we have
$$-a\leq W(r)\leq a$$
Then Lemma \ref{borne} implies that on this interval we have
$$-a \leq \p_rW(r)\leq a.$$ 
$W$ is initially decreasing and can not have a local minimum in the
region $W>0$ (this is a consequence of the maximum principle, see Lemma \ref{max}). Consequently there exists $r_1>1$ such that $0\leq W\leq
a$ on $[1,r_1]$ and $W(r_1)=0$. Because of the bound of
  the derivative we have $r_1\ge 2$. By the same bound  we have $-a\leq W\leq a$ on $[r_1,r_1+1]$.
\end{proof}

Let $Q(r)=1-\frac{1}{r}-\frac{1}{2r^2}$.
\begin{proposition}\label{enc}
We have for $r\geq 3$
$$-Q(r)\leq W(r)\leq Q(r)$$
\end{proposition}

Let 
$$L(u,r)= \left(1-\frac{1}{r}\right)\p_r^2u+\frac{1}{r^2}\p_ru+\frac{1}{r^2}u(1-u^2).$$
Before proving Proposition \ref{enc}, we need the following lemma
\begin{lemma}
For $r\geq 3$ we have $L(Q,r)<0$ and $L(-Q,r)>0$.
\end{lemma}
\begin{proof}
Since $L$ is odd in $u$, it is sufficient to prove $L(Q,r)<0$. We calculate
\begin{align*}L(Q,r)=&\left(1-\frac{1}{r}\right)\left(-\frac{2}{r^3}-\frac{3}{r^4}\right)+\frac{1}{r^2}\left(\frac{1}{r^2}+ \frac{1}{r^3}\right)
+\frac{1}{r^2}\left(1-\frac{1}{r}-\frac{1}{2r^2}\right)\left(1-\left(1-\frac{1}{r}-\frac{1}{2r^2}\right)^2\right)\\
=&-\frac{2}{r^4}+\frac{2}{r^5}+ \frac{3}{4r^6}+ \frac{3}{4r^7}+ \frac{1}{8r^8}.
\end{align*}
Consequently, for $r\geq 3$ we have
$$L(Q,r)\leq \frac{1}{r^4}\left(-2+\frac{2}{3}+\frac{3}{4*3^2}+\frac{3}{4*3^3 }+\frac{1}{8*3^4}\right) \leq -\frac{1}{r^4}<0.$$
\end{proof}

\begin{proof}[Proof of Proposition \ref{enc}]
We have $-a\leq W(3) \leq a$ and 
$$a<\frac{11}{18}=1-\frac{1}{3}-\frac{1}{2*9}=Q(3).$$
If the inequality of Proposition \ref{enc} is false, there exists $r_1<r_2$ with $r_2$ which can be infinite such that
$$W(r_1)=Q(r_1), \quad W(r_2)=Q(r_2)$$
and $W>Q$ on $]r_1,r_2[$ (The case $W<-Q$ is treated in a similar way). Consider $r_0$ such that $W-Q$ is maximum at
$r_0$. Note that such a maximum always exists
  independently if $\lim_{r\rightarrow \infty} W(r)=-1$ (in which case
  $r_2<\infty$) or $\lim_{r\rightarrow
    \infty}W(r)=1=\lim_{r\rightarrow \infty}Q(r)$.  Then we have
$$L(W,r_0)-L(Q, r_0)=-L(Q, r_0)>0$$
so
$$\left(1-\frac{1}{r_0}\right) (\p_r^2W-\p_r^2Q)(r_0)+ \frac{1}{r_0^2}\left(W(1-W^2)-Q(1-Q^2)\right)>0$$
Since
$$W(r_0)>Q(r_0)\geq Q(3)= \frac{11}{18} \geq \frac{1}{\sqrt{3}}$$
and the function $x\mapsto x(1-x^2)$ is decreasing for $x\geq \frac{1}{\sqrt{3}}$ we have 
$$\left(W(1-W^2)-Q(1-Q^2)\right) \leq 0$$ and consequently
$$\left(1-\frac{1}{r_0}\right) (\p_r^2W-\p_r^2Q)(r_0) >0$$
which is a contradiction with the fact that $W-Q$ is maximum at
$r_0$. 
\end{proof}

\subsubsection{A bound on the potential}
We now come back to the potential 
$$V=P(3W^2-1)$$
\begin{proposition}
We have
$$\int_{\R} V(x) dx <0.$$
\end{proposition}
\begin{proof}
First note that 
\begin{equation*}
\int_{\R} V(x)dx=\int_1^{\infty} \frac{3W^2-1}{r^2} dr. 
\end{equation*}
We estimate
$$\int_1^3 \frac{3W^2-1}{r^2} \leq \int_1^3 \frac{3a^2-1}{r^2}= \frac{2(3a^2-1)}{3}$$
and
\begin{eqnarray*}
\int_3^\infty \frac{3W^2-1}{r^2} &\leq& \int_3^\infty \frac{1}{r^2}\left( 3\left(1-\frac{1}{r}\right)^2-1\right)
\leq \int_3^\infty \frac{1}{r^2}\left(2-\frac{6}{r}+
  \frac{3}{r^2}\right)=\left[-\frac{2}{r}+\frac{3}{r^2}-\frac{1}{r^3}
\right]_3^\infty\\
&=& \frac{1}{3}+ \frac{1}{27}= \frac{10}{27}
\end{eqnarray*}
Note that
$$\frac{2(3a^2-1)}{3}+\frac{10}{27}<0$$
because $a\le\frac{1+\sqrt{3}}{5+3\sqrt{3}}<
\frac{2}{3\sqrt{3}}$. Therefore we have
$$\int_{\R} V(x) dx<0.$$
\end{proof}

\subsection{Proof of Theorem \ref{Mainth}}
The main theorem with $\cE$ replaced by $X$ now follows from the
abstract result. In order to be able to replace $X$ by $\cE$ we need
the following lemma. We will drop the index $n$. 
\begin{lemma}
\label{lem2}
Let $\phi_0$ be an eigenfunction of $\cA$ with eigenvalue $-\lambda^2$. Then we
have 
\begin{eqnarray}
\label{lem2.1}
\int_{\R} P\vert\phi_0\vert^2\ge \lambda^2\int_{\R} \vert
\phi_0\vert^2.\\
\label{lem2.2}
-\int V\vert\phi_0\vert^2\ge 0. 
\end{eqnarray}
\end{lemma}
\proof 

Let us first show \eqref{lem2.1}. We have 
\begin{equation*}
(-\partial_x^2+V)\phi_0=-\lambda^2\phi_0. 
\end{equation*}
Multiplication by $\phi_0$ and integration by parts gives 
\begin{equation}
\label{lem2.3}
\int \vert \phi_0'\vert^2+\int V\vert \phi_0\vert^2+\lambda^2\int
\vert \phi_0\vert^2=0.
\end{equation}
Now recall that $V=P(3W^2-1)$, thus 
\begin{equation*}
\int P\vert \phi_0\vert^2\ge \lambda^2\int \vert\phi_0\vert^2. 
\end{equation*}
We now show \eqref{lem2.2}. From \eqref{lem2.3} we obtain :
\begin{equation*}
-\int V\vert \phi_0\vert^2=\int \vert\phi_0'\vert^2+\lambda^2\int
\vert\phi_0\vert^2\ge 0.
\end{equation*}
\qed

Let $\tilde{\cH}^1$ the completion of $C_0^{\infty}$ for
the norm 
\begin{equation*}
\Vert u\Vert_{\tilde{\cH}^1}^2=\Vert u\Vert_{\dot{H}^1}^2+\Vert
u\Vert_{L^2_P}^2
\end{equation*}
We put $\tilde{\cE}=\tilde{\cH}^1\times L^2$. 

{\bf Proof of Theorem \ref{Mainth}}

We continue using the notations of the abstract setting. We claim that it is sufficient to show the following :
\begin{equation}
\tag{IM}\label{IM}
\begin{array}{c} \mbox{There exists $\epsilon_0>0$ and a sequence $\psi_0^m$ with $\Vert
\psi_0^m\Vert_X\rightarrow 0,\, m\rightarrow \infty$,}\\ \mbox{but for all $m$}\quad 
\sup_{t\ge 0}\Vert \psi^m(t)\Vert_{\tilde{\cE}}\ge
\epsilon_0>0.\end{array}
\end{equation}
To see this we first note that 
\begin{equation*}
\Vert \psi_0^m\Vert_{\cE}\le \Vert \psi_0^m\Vert_X
\end{equation*}
because 
\begin{equation*}
\left(\int P\vert u\vert^4\right)^{1/4}\lesssim
\Vert u\Vert^{1/2}_{\infty}\Vert u\Vert_{L^2}^{1/2}\le \Vert u\Vert_{H^1}
\end{equation*}
by the Sobolev embedding $H^1\hookrightarrow L^{\infty}$. On the other hand 
\begin{equation*}
\Vert u\Vert_{L^2_P}=\left(\int P\vert u\vert^2\right)^{1/2}\le
\left(\int P\right)^{1/4}\left(\int P\vert
  u\vert^4\right)^{1/4}\lesssim \Vert u\Vert_{L^4_P}
\end{equation*}
and thus 
\begin{equation*}
\Vert \psi^m(t)\Vert_{\cE}\gtrsim \Vert \psi^m(t)\Vert_{\tilde{\cE}}.
\end{equation*}
Let us now show \eqref{IM}. We follow the proof of the main
theorem. We choose 
\begin{equation*}
\psi_0=\left(\begin{array}{c} \phi_0 \\ \frac{\lambda_0}{i}
    \phi_0\end{array}\right),\, \phi_0\in
  \one_{\{-\lambda_0^2\}}(\cA)\cH,\, \Vert
  \phi_0\Vert=\frac{1}{3(1+\Vert V_-\Vert_{\infty})^{1/2}}\rho.  
\end{equation*}
We estimate 
\begin{eqnarray*}
\Vert
\psi_0\Vert_X^2&=&\<(-\partial_x^2+V)\phi_0,\phi_0\>-\<V\phi_0,\phi_0\>+\Vert
\phi_0\Vert^2+\lambda_0^2\Vert \phi_0\Vert^2\\
&\le&(\Vert V_{-}\Vert_{\infty}+1)\Vert\phi_0\Vert^2=1/9\rho^2. 
\end{eqnarray*}
Thus the first part of the proof goes through without any changes. We
then have to estimate $\Vert \psi(\tau)\Vert_{\tilde{\cE}}$. We estimate 
\begin{eqnarray*}
\Vert
\psi_0\Vert^2_{\tilde{\cE}}&=&\<\cA\phi_0,\phi_0\>-\<V\phi_0,\phi_0\>+\int
P\vert\phi_0\vert^2+\lambda_0^2\vert \phi_0\vert^2\\
&=&-\<V\phi_0,\phi_0\>+\int P\vert\phi_0\vert^2\\
&\ge&\int P\vert \phi_0\vert^2\ge \lambda_0^2\int \vert\phi_0\vert^2\\
&=&\lambda_0^2\frac{1}{9(1+\Vert V_-\Vert_{\infty})}\rho^2.
\end{eqnarray*}
Here we have used Lemma \ref{lem2}. Using 
\begin{equation*}
\Vert u\Vert_{\tilde{
\cE}}\le C_1\Vert u\Vert_X
\end{equation*} 
we find 
\begin{eqnarray*}
\Vert \psi(\tau)\Vert_{\tilde{\cE}}\ge \frac{\lambda_0}{3(1+\Vert
  V_-\Vert_{\infty})^{1/2}}\rho-\frac{2C_1M_LM_F\Vert
  E_1\Vert}{2\beta-\lambda_1}\rho^2\ge \frac{\lambda_0}{6(1+\Vert V_-\Vert_{\infty})}\rho
\end{eqnarray*}
for $\rho$ small enough. 
\qed
	
\subsection{Proof of Corollary \ref{corstat}}
We recall
$$E^{(\p_t)}(F(t))=
\cE(W,\dot{W}).$$
We take the same sequence of data $W_{0,n}^m$ as in Theorem
\ref{Mainth}. We first have to show that 
\begin{equation*}
\int P ((W_{0,n}^m)^2-W_n^2)^2\rightarrow 0,\quad m\rightarrow \infty. 
\end{equation*} 
This follows from 
\begin{eqnarray*}
\int P ((W_{0,n}^m)^2-W_n^2)^2&\lesssim&\int P (W^m_{0,n}-W_n)^4+\int
PW_n^2(W^m_{0,n}-W_n)^2\\
&\lesssim&\int P (W^m_{0,n}-W_n)^4+\left(\int P (W^m_{0,n}-W_n)^4\right)^{1/2}\rightarrow 0,\quad m\rightarrow \infty
\end{eqnarray*}
by Theorem \ref{Mainth}. In the first inequality we have used the estimate
$$(A^2-B^2)^2= (A-B)^2(A+B)^2= (A-B)^2(A-B+2B)^2 \leq 2(A-B)^4+ 8B^2(A-B)^2,$$
and the fact that $\Vert W_n \Vert_{L^\infty}\leq 1$.
Now we have to show that 
\begin{equation}
\label{eqcor1}
\sup_{t\ge 0} \int (\dot{W}^m_n)^2+((W^m_n)'-W_n')^2+P
((W^m_n)^2-W_n^2)^2\ge \epsilon_1>0.
\end{equation}
We know by Theorem \ref{Mainth} that 
\begin{equation}
\label{eqcor2}
\sup_{t\ge 0}\int (\dot{W}^m_n)^2+((W^m_n)'-W_n')^2+P
(W^m_n-W_n)^4\ge \epsilon_0>0.
\end{equation}
We also know from the proof of Theorem \ref{Mainth} that this supremum is achieved on the interval
$[0,m]$ and that on this interval 
\begin{equation*}
\Vert W^m_n-W_n\Vert_{L^2}\le \rho
\end{equation*}
for some $\rho>0$. Now observe that for $u\in H^1$ we have 
\begin{equation}
\label{eqcor3}
\int Pu^2\le 2\left(\int \frac{1}{r^2}
  u^2\right)^{1/2}\left(\int (u')^2\right)^{1/2}. 
\end{equation}
Indeed by density we can suppose $u\in C_0^{\infty}(\R)$ and then
compute 
\begin{eqnarray*}
\int Pu^2=\int \p_x(-\frac{1}{r})u^2=2\int \frac{1}{r}uu'\le 2 \left(\int \frac{1}{r^2}
  u^2\right)^{1/2}\left(\int (u')^2\right)^{1/2}. 
\end{eqnarray*}
Let us now show \eqref{eqcor1}. We can suppose that 
\begin{equation*}
\sup_{t\ge 0} \int ((W_n^m)'-W_n')^2\le \frac{\epsilon^2_0}{4^4\rho^2}
\end{equation*}
because otherwise there is nothing to show. Then we estimate 
\begin{eqnarray*}
\lefteqn{\int \dot{W_n^m}^2+((W_n^m)'-W'_n)^2+P((W_n^m)^2-W_n^2)^2}\\
&\ge&\int \dot{W_n^m}^2+((W_n^m)'-W'_n)^2+\frac{1}{2}P(W_n^m-W_n)^4-4PW_n^2(W_n^m-W_n)^2\\
&\ge&\int\frac{1}{2}(\dot{W_n^m}^2+((W_n^m)'-W'_n)^2+P(W_n^m-W_n)^4)\\
&-&4\left(\int((W_n^m)'-W_n')^2\right)^{1/2}\left(\int\frac{(W_n^m-W_n)^2}{r^2}\right)^{1/2}\\
&\ge&\frac{1}{2}\int\dot{W_n^m}^2+((W_n^m)'-W'_n)^2+P(W_n^m-W_n)^4-\epsilon_0/4,
\end{eqnarray*}
where in the first inequality we have used the estimate
\begin{align*}(A^2-B^2)^2&= (A-B)^2(A-B+2B)^2 = (A-B)^2\left((A-B)^2 +4B(A-B) +4B^2\right) \\
&\geq (A-B)^2((A-B)^2-\frac{1}{2}(A-B)^2-8B^2+4B^2)=
 \frac{1}{2}(A-B)^4-4B^2(A-B)^2,
 \end{align*}
 and in the second inequality we have used \eqref{eqcor3} and the fact that $\Vert W_n \Vert_{L^\infty}\leq 1$.

The supremum over $t\ge 0$ of this expression is $\ge \epsilon_0/4$ by
\eqref{eqcor2}.
\qed
\appendix
\section{Proof of Theorem \ref{thstat}}
In this appendix we give an explicit proof of theorem \ref{thstat}. We
adapt in the simpler uncoupled case the arguments of Smoller Wasserman
Yau, McLeod \cite{SWYM}; Smoller, Wasserman, Yau \cite{SWY} and
Smoller, Wasserman \cite{SW} to show the existence of infinitely many
solutions. In this appendix we work with the $r$ variable and
\underline{we note $'=\p_r$} in this appendix !  Again we
  can suppose that $m=1/2$. Recall that the
stationary equation writes
\begin{equation}
\label{eqA.1}
\left(1-\frac{1}{r}\right)W''+\frac{1}{r^2}W'+\frac{1}{r^2}W(1-W^2)=0.
\end{equation}
\subsection{Local solutions}
\begin{proposition}
Let $0<\alpha<1$ and $0\leq a \leq 1$. There exists $r_a>1$ and a unique solution $W \in C^{2,\alpha}([1,r_a])$ with boundary condition
$$W(1)=a, \quad W'(1)= b , \quad W''(1)=c$$
where
$$b= -a(1-a^2), \quad 2c=-b(1-3a^2)$$
\end{proposition}
\begin{proof}
We set $z= W'$ to write the equation as a first order system. We consider
$$X= \{(w,z)\in C^{(2,\alpha) }([1,1+\ep])\times C^{(1,\alpha) }([1,1+\ep]), w(1)=a,w'(1)=z(1)=b,w''(1)=z'(1)=c\}$$
and the map $T:(w,z) \in X \mapsto (\wht w,\wht z)$ with
\begin{align*}
\wht w &= a+ \int_1^r z,\\
\wht z &= b- \int_1^r \frac{1}{\rho(\rho-1)}(z+w(1-w^2)).
\end{align*}
We first show that $T$ preserves the boundary conditions.
We calculate
\begin{align*}\wht z'=& -\frac{1}{r(r-1)}(z+w(1-w^2))\\
=&-\frac{1}{r(r-1)}\left(z
+a(1-a^2) + \int_1^r w'(1-3w^2)
\right)\\
=&-\frac{1}{r(r-1)}(z-b)  -\frac{1}{r(r-1)}\int_1^r w'(1-3w^2)
\end{align*}
so $\wht z'(r)\rightarrow -z'(1) -w'(1)(1-3w^2(1))=-c+2c=c$ when $r\rightarrow 1$.
We now show that $T$ is a contraction in $B_X(0,A)$ for $\ep$ small enough. For this the only difficulty is to estimate
\begin{align*}
\frac{|\wht z'(r)-\wht z'(1)|}{|r-1|^\alpha}
\le & \frac{1}{|r-1|^\alpha}\left|-\frac{1}{r(r-1)}(z+w(1-w^2))-c\right|\\
\le&\frac{1}{|r-1|^\alpha}\Big|-\frac{1}{r(r-1)}\Big(
b+\int_1^r (z'(\rho)-z'(1))d\rho+c(r-1)\\
&+a(1-a^2)+\int_1^r
(w(1-w^2))'(\rho)-(w(1-w^2))'(1) d\rho+b(1-3a^2)(r-1)\Big)-c\Big|
\\
\le &\frac{1}{|r-1|^\alpha}\left|-\frac{1}{r(r-1)}
\int_1^r (z'(\rho)-z'(1))d\rho\right|\\
+& \frac{1}{|r-1|^\alpha}\left| -\frac{1}{r(r-1)} \left(-(r-1)c+ \int_1^r
(w(1-w^2))'(\rho)-(w(1-w^2))'(1) d\rho\right)-c
\right|\\
\le & \frac{1}{r(r-1)^{1+\alpha}}\int_1^r |z'(\rho)-z'(1)|
+c\frac{1}{(r-1)^{\alpha}}\left(1-\frac{1}{r}\right)\\
+&\frac{1}{(r-1)^{1+\alpha}}\|(w(1-w^2))'\|_{C^{1}} (r-1)^2\\
\le & \frac{1}{r(r-1)^{1+\alpha}}\|z'\|_{C^{0,\alpha}}\int_1^r |\rho-1|^\alpha + c\ep^{1-\alpha}+ C(\|w\|_{C^2}) \ep^{1-\alpha}\\
\le &\frac{1}{1+ \alpha}\|z'\|_{C^{0,\alpha}}
+c\ep^{1-\alpha}+ C (\|w\|_{C^{2}})\ep^{1-\alpha}.
\end{align*}
Consequently we can show that for $\ep$ small enough, $T$ is a contraction, with contracting constant $\frac{1}{1+\alpha}+C(A)\ep^{1-\alpha}$, and consequently it has a unique fixed point.
\end{proof}

As a corollary of the proof of local existence we obtain the continuity of the family of solutions $W_a$ with respect to the initial data $a$.

\begin{corollary}\label{cont}
Let $\delta>0$. If $W_a$ is a solution on $[1,R]$ with $-1\leq W_a\leq1$ and $a'$ is sufficiently close to $a$, then $W_{a'}$ is defined on $[1,R]$ we have 
$$\|W_a-W_{a'}\|_{C^{2,\alpha}([1,R])} \leq \delta.$$
\end{corollary}

\subsection{Basic facts}
\begin{lemma}\label{borne}
Let $0<B\leq 1$. As long as $W$ is a $C^2$ solution with $|W|\leq B$ we have $|W'|\leq B$.
\end{lemma}

\begin{proof}
Assume that in $[1,r_0]$ we have $|W|\leq B$.
Then 
$$-\frac{B}{r^2}\leq (1-\frac{1}{r})W''+ \frac{1}{r^2}W'\leq \frac{B}{r^2}$$
and consequently
$$\left[\frac{B}{r}\right]_1^r \leq \left[(1-\frac{1}{r})W'\right]_1^r \leq \left[-\frac{B}{r}\right]_1^r$$
so
$$-B \leq W'(r)\leq B.$$
\end{proof}

\begin{corollary}
The solution $W$ exists and is $C^{2,\alpha}$ as long as $|W|\leq 1$.
\end{corollary}

We now consider the solution $W$ on $[0,r_a[$ where $r_a$ is the smallest $r$ such that $|W|=1$ if it exists, and $r_a= \infty$ otherwise.

\begin{lemma}\label{max}
The solution $W$ cannot have a local minimum with $W>0$ nor a local maximum with $W<0$.
\end{lemma}
\begin{proof}
If $W$ has a positive local minimum at $r_0$ then
$$\left(1-\frac{1}{r_0}\right)W''(r_0)+\frac{1}{r^2}W(r_0)(1-W^2(r_0))=0$$
but $\frac{1}{r^2}W(r_0)(1-W^2(r_0))>0$ (the local minimum cannot be $1$), and $W''(r_0)\geq 0$, which is a contradiction.
\end{proof}

\begin{lemma}\label{lim}
The solution $W$ can not have a limit $l \neq -1,0,1$.
\end{lemma}

\begin{proof}
Assume that $W\rightarrow l$ with $0<l<1$. We can write for $r_1$ big enough and $r_n \geq r_1$
$$\left[\frac{l(1-l^2)+\ep}{r}\right]_{r_1}^{r_n} \leq \left[(1-\frac{1}{r})W'\right]_{r_1}^{r_n} \leq \left[\frac{l(1-l^2)-\ep}{r}\right]_{r_1}^{r_n}.$$
Since $W \rightarrow l$ there exists a sequence $r_n \rightarrow \infty$ such that $W'(r_n) \rightarrow 0$. Letting $n\rightarrow \infty$ we obtain
$$   \frac{l(1-l^2)-\ep}{r_1} \leq W'(r_1)\left(1-\frac{1}{r_1}\right) \leq \frac{l(1-l^2)+\ep}{r_1}$$
so
$$   \frac{l(1-l^2)-\ep}{r_1-1} \leq W'(r_1)\leq \frac{l(1-l^2)+\ep}{r_1-1}$$
and there exists a constant $C$ such that for $r$ big enough
$$(l(1-l^2)-\ep)\ln(r-1) \leq W(r)-C$$
which is a contradiction.
\end{proof}

\subsection{More technical facts}
\begin{proposition}\label{ham}
Let $0\leq a \leq 1$. There exists $\ep>0$ and $R>0$ such that if
there exists $R<r_0<r_a$ such that $W$ has a local extremum at $r_0$ with $1-\ep \leq |W(r_0)|<1$ then $r_a<\infty$ and $W$ has one and only one zero in $[r_0,r_a]$. 
\end{proposition}

\begin{proof}
We consider the case $W(r_0)>0$. The other case can be treated similarly.
We consider
$$H= r^2\frac{(W')^2}{2} + \frac{W^2}{2}-\frac{W^4}{4}.$$
We calculate
\begin{align*}H'(r)=&r (W')^2 +r^2\frac{1}{r(1-r)}(W'+W(1-W^2))W' + WW'-W^3W'\\
= &(W')^2\left(r +\frac{r^2}{r(1-r)}\right)+ WW'(1-W^2)\left(1+ \frac{r^2}{r(1-r)}\right)\\
=&(W')^2\left(r +\frac{r^2}{r(1-r)}\right)+ WW'(1-W^2)\frac{1}{1-r}.
\end{align*}
Let $R$ be such that for $r\geq R$ we have
$$\left(r +\frac{r^2}{r(1-r)}\right)>0$$
then for $r\geq R$ and $WW'\leq 0$ we have $H'(r)>0$. With our assumption on $r_0$ we can estimate
$$H(r_0) \geq \frac{(1-\ep)^2}{2}-\frac{1}{4}\geq \frac{1}{4\delta}$$
for a suitable $\delta$ which will be precised later, and $\ep$ small enough.
Since $W$ has a local maximum at $r_0$, there are two possibilities
\begin{itemize}
\item We have $r_a= \infty$, $W$ is decreasing on $[r_0,+ \infty[$ and $W\rightarrow 0$ at $\infty$.
\item There exists $r_1<r_a$ such that $W(r_1)= 0$ and W is decreasing on $[r_0,r_1]$.
\end{itemize}
In the first case we obtain for all $r\geq r_a$, $H'(r)>0$ so
$$H(r) \geq \frac{1}{4\delta}$$
and since $W\rightarrow 0$ the expression of $H$ yields the existence of $r_2$ such that for $r\geq r_2$
$$W'(r)^2 \geq \frac{1}{(2\delta+1)r^2}$$
so
$W'(r)\leq -\frac{1}{\sqrt{2\delta+1}r}$
and $W(r) \leq W(r_2)-\frac{1}{\sqrt{2\delta+1}}\ln(r)$ which is a contradiction. 
Consequently we are in the second case.
We have $H(r_1)\geq H(r_0)$ so we can estimate $W'(r_1)$
$$W'(r_1)\leq -\frac{1}{\sqrt{2\delta}r_1}$$
Moreover, when $-1\leq W\leq 1$ we have
$W(1-W^2)\leq \frac{2}{3\sqrt{3}}$ and consequently we can write for $r_a>r_2>r_1$
$$\left[W'(r)\left(1-\frac{1}{r}\right)\right]_{r_1}^{r_2}
\leq \left[-\frac{2}{3\sqrt{3}r}\right]_{r_1}^{r_2}$$
and consequently
$$\left(1-\frac{1}{r_2}\right)W'(r_2)\leq
-\left(1-\frac{1}{r_1}\right)\frac{1}{\sqrt{2\delta}r_1}
+\frac{2}{3\sqrt{3}r_1}-\frac{2}{3\sqrt{3}r_2}.$$
For $r_1$ big enough (which is possible by choosing $R$ big enough) and $\delta$ close enough to $1$ (which is possible by choosing $\ep$ small enough) we have
$$-\left(1-\frac{1}{r_1}\right)\frac{1}{\sqrt{2\delta}r_1}
+\frac{2}{3\sqrt{3}r_1}\leq 0,$$
since $\frac{2}{3\sqrt{3}}<\frac{1}{\sqrt{2}}$.
and consequently
$$W'(r_2)\leq -\frac{2}{3\sqrt{3}(r_2-1)},$$
and therefore $r_a<\infty$, $W(r_a)=-1$ and $W'$ is decreasing on $[r_1,r_a]$. This concludes the proof of Proposition \ref{ham}
\end{proof}

\begin{proposition}\label{zero}
Let $N>0$. Then for $a$ small enough, the solution $W$ has more than $N$ zeros on $[1,r_a]$
\end{proposition}

\begin{proof}
To count the number of zeros of a solution $W$ we can introduce the function $\theta$ which is the continuous function such that
$$\tan(\theta)= \frac{W'}{W}$$
and $ -\frac{\pi}{2}<\theta(1)<\frac{\pi}{2}$. Then $W$ has $N$ zero between $1$ and $r_0$ if and only if
$$ -\frac{\pi}{2}-N\pi<\theta(r_0)<\frac{\pi}{2}-N\pi.$$
It is totally similar to count the number of zero thanks to the function $\psi$ defined by
$$\tan(\psi)= \frac{rW'}{W}.$$
We estimate $\psi'$
\begin{align*}\psi'(r)=&\frac{1}{1+\left(\frac{rW'}{W}\right)^2}\frac{W(W'+rW'')-r(W')^2}{W^2}\\
=&\frac{WW'-W\frac{1}{r-1}(W'+W(1-W^2)) -(rW')^2}{W^2+(rW')^2}\\
=&\frac{WW'\frac{r-2}{r-1}-\frac{1}{r-1}W^2(1-W^2) -(rW')^2}{W^2+(rW')^2}.
\end{align*}
We first estimate
$$\frac{WW'\frac{r-2}{r-1}}{W^2+(rW')^2}\leq \frac{1}{r}\frac{|rWW'|}{W^2+(rW')^2}\leq \frac{1}{2r}.$$
We assume that $|W|\leq \delta$.
To estimate the other terms we consider three cases
\begin{itemize}
\item $2|W|^2\leq |rW'|^2$. Then we have
$$\psi'(r) \leq \frac{1}{2r} -\frac{r(W')^2}{W^2+ (rW')^2}
\leq \frac{1}{2r}-\frac{2}{3r}= -\frac{1}{6r}.$$
\item $2|rW'|^2\leq |W|^2$. Then we have
$$\psi'(r) \leq \frac{1}{2r} -\frac{W^2(1-W^2)}{(r-1)(W^2+(rW')^2)}
\leq \frac{1}{2r}-\frac{2(1-\delta^2)}{3(r-1)}
\leq \frac{3-4(1-\delta^2)}{6r}.$$
\item $\frac{1}{2}|W|^2\leq|rW'|^2\leq 2|W|^2$. Then we have
$$\psi'(r) \leq \frac{1}{2r} -\frac{r(W')^2}{W^2+ (rW')^2}
 -\frac{W^2(1-W^2)}{(r-1)(W^2+(rW')^2)}
 \leq \frac{1}{2r}-\frac{1}{3r}
-\frac{(1-\delta^2)}{3r}\leq \frac{3-4+2\delta^2}{6r}.$$
\end{itemize}
If we take $\delta$ small enough we then have
$$\psi'(r)\leq -\frac{1}{12r}.$$
Let now $R$ be such that $-\frac{1}{12}\ln(R)\leq -N\pi$.
Thanks to Corollary \ref{cont}, we can find $a_0$ small enough such that for $0\leq a \leq a_0$ the solution exists on $[1,R]
$ and satisfies $|W|\leq \delta$ on this interval. Then $\psi(R)-\psi(1)\leq -N\pi$  so $W$ has at least $N$ zero on $[1,R]$. This concludes the proof of Proposition \ref{zero}.
\end{proof}
\begin{corollary}\label{limbis}
Let $W$ be a solution with $r_a= \infty$ and a finite number of zeros. Then $W \rightarrow \pm 1$.
\end{corollary}

\begin{proof}
Because of Lemma \ref{max} a solution with a finite number of zeros has a finite limit. Because of Lemma \ref{lim} this limit must be $0$ or $\pm 1$.  If it was $0$ we could find $R_0$ such that $|W|\leq \delta$ for $r\geq R_0$, with $\delta$ defined in the proof of Proposition \ref{zero}. then for $r\geq R_0$ we have
$$\psi'(r)\leq -\frac{1}{12r},$$
consequently $\psi$ is unbounded from above, so $W$ has an infinite number of zeros.
\end{proof}

\subsection{Proof of the Theorem}

\begin{lemma}
Let $X_n$ be the set of initial data $a$ such that the corresponding solution has $n$ zeros and satisfies $r_a<\infty$. Then $X_n$ is open and if $\alpha$ is a limit point of $X_n$ the corresponding solution satisfies $r_\alpha= \infty$ has $m$ zeros, with $m=n$ or $m=n-1$ and tends to $(-1)^m$ at infinity.
\end{lemma}

\begin{proof}
The fact that $X_n$ is open is a direct consequence of Corollary \ref{cont}. Let $\alpha$ be a limit point and let $a_i \in X_n$ be such that $a_i \rightarrow \alpha$. 

Assume first that $W_\alpha$ is such that $r_a<\infty$. Then we can compare the solution on the fixed interval $[1,r_\alpha]$ so Corollary \ref{cont} implies that $W_\alpha$ has exactly $n$ zeros, so $\alpha \in X_n$ which is a contradiction. Consequently $r_a= \infty$. This also implies that the sequence of $r_{a_i}$ is not bounded.

Assume now that there exists $R$ such that $W_\alpha$ has strictly more than $n$ zeros before $R$. Once again Corollary \ref{cont} yields a contradiction. 

Assume now that $W_{\alpha}$ has $m$ zeros with $m<n$. Then thanks to Corollary \ref{limbis} $W_{\alpha}$ tend to $(-1)^m$ and Proposition \ref{ham} implies that the $W_{a_i}$ have $m$ or $m+1$ zeros.
\end{proof}

\begin{proof}[Proof of Theorem \ref{thstat}]
Let $\wht X_n$ be the set of initial data $a$ such that the corresponding solution has less than $n$ zeros.
Let $\alpha = \min(\wht X_n)$. Proposition \ref{zero} implies that $\alpha>0$.
There are two case
\begin{itemize}
\item If $\alpha \in \wht X_n$, then $W_{\alpha}$ is a solution with $m\leq n$ zeros with $r_{\alpha}= \infty$. Then Corollary \ref{cont} and Proposition \ref{ham} imply that for $a$ close to $\alpha$ either the solutions have $m$ zeros, either have $m+1$ zeros and $r_a<\infty$. Consequently we have $m=n$ and considering a sequence $a_i<\alpha$ converging to $\alpha$ we have shown that $X_{n+1}$ is non empty.

\item If $\alpha \notin \wht X_n$ then $W_{\alpha}$ must be a solution with $r_{\alpha}= \infty$ and $k>n$ zeros. But we have shown in the previous point that in a neighborhood of such a solution we can only have solutions with $k$ or $k+1$ zeros, so this case can not occur. 

\end{itemize}

We start the iteration with the function
$$W_1 = \frac{c-r}{r+3(c-1)}, \quad c=\frac{3+\sqrt{3}}{2},$$
which is a special solution of \eqref{ym}, with only one $0$ (see
\cite{BCC}). Note also that 
$$W_1(1)= \frac{1+\sqrt{3}}{5+3\sqrt{3}}=a_1.$$ We then obtain at least one solution $-1\leq W_{a_n} \leq 1$ for each number of zeros $n$. This concludes the proof of Theorem \ref{thstat}.
\end{proof}

\end{document}